\documentclass[12pt,reqno]{amsart}

\newcommand{\T}{{\mathbf T}^m}

\newcommand{\kahler}{K\"ahler }

\newcommand{\PP}{{\mathbb P}}
\newcommand{\R}{{\mathbb R}}
\newcommand{\C}{{\mathbb C}}

\newcommand{\Z}{{\mathbb Z}}

\newcommand{\CP}{\C\PP}

\renewcommand{\phi}{\varphi}

\newcommand{\lcal}{\mathcal{L}}

\def    \Z  {{\mathbb Z}}
\def    \T  {{\mathbb T}}
\def    \R  {{\mathbb R}}
\def    \C  {{\mathbb C}}

\usepackage{color}
\usepackage{amsthm}
\usepackage{amsmath}
\usepackage{amstext}
\usepackage{amssymb}
\usepackage{amsfonts}
\usepackage{latexsym}
\usepackage{amscd}

\newtheorem{maintheo}{{\sc Theorem}}

\newtheorem{theo}{{\sc Theorem}}[section]

\newtheorem{remark}[theo]{{\sc Remark}}
\newtheorem{lem}[theo]{{\sc Lemma}}
\newtheorem{prop}[theo]{{\sc Proposition}}

\newenvironment{defin}{\medskip\noindent{\it Definition:\/} }{\medskip}

\begin{document}
\title[Szasz Analytic Functions and Noncompact K\"{a}hler Toric Manifolds]
{Szasz Analytic Functions and Noncompact K\"{a}hler Toric Manifolds}
\author[Renjie Feng]{Renjie Feng}
\address{Department of Mathematics, Johns Hopkins University, USA}
\email{rfeng@math.jhu.edu}
\date{\today}
\maketitle
\begin{abstract} We show that the classical Szasz analytic function $S_N(f)(x)$ is obtained
by applying the pseudo-differential operator $f(N^{-1}D_{\theta})$ to the Bergman kernels  for the Bargmann-Fock space. The expression generalizes immediately to any smooth  polarized noncompact complete toric \kahler manifold, defining the generalized Szasz analytic function $S_{h^N}(f)(x)$.  About $S_{h^N}(f)(x)$, we prove that it admits complete asymptotics and there exists a universal scaling limit. 
  As an example, we will further compute $S_{h^N}(f)(x)$ for the Bergman metric on the unit ball. 
 \end{abstract}
\section*{Introduction}
\subsection{Introduction}
In \cite{Z1}, S. Zelditch relates the classical Bernstein polynomials \cite{B} to the Bergman kernel for the Fubini-Study metric on $\CP^m$, and then generalizes this relation to any smooth polarized compact \kahler toric manifold. In this article, we will further generalize the construction to any smooth polarized noncompact complete toric \kahler manifolds, and introduce the generalized Szasz analytic function. In the model case of the Bargmann-Fock space (see subsection \ref{bf}), the generalized Szasz analytic function is the classical Szasz analytic function \cite{S}: For any $f\in C^{\infty}(\R^m)$,
 \begin{equation} \label{SZASZ} S_{h^N_{BF}}(f)(x) = e^{-N\|x\|}\sum _{\alpha \in \Z_+^m}f(\frac{\alpha}{N})\frac{(Nx)^{\alpha}}{\alpha!}\end{equation}
 where $\|x\|=x_1+\cdots+x_m$, $\alpha=(\alpha_1,\ldots, \alpha_m)$ is a lattice point in $\Z_+^m$, $\alpha!=\alpha_1!\cdots \alpha_m!$ and $x^{\alpha}=x_1^{\alpha_1}\cdots x_m^{\alpha_m}$. In \cite{S}, O. Szasz proved that if $f(x)$ is a smooth Schwartz function, then $S_{h^N_{BF}}(f)(x)$ converges to $f(x)$ uniformly on any compact subset of $\R_{+}^m$, i.e., he generalized the Bernstein polynomials to the infinite interval.

In section \ref{F}, we show that the classical Szasz analytic function can be expressed in term of the the Bargmann-Fock space in two ways. Let $e^{i\theta}=(e^{i\theta_1},\ldots,e^{i\theta_m})$ be the standard $\mathbb{T}^m$
action on $\C^m$ and denote $D_{\theta}=(\frac{1}{i}\frac{\partial}{\partial\theta_{1}},\ldots, \frac{1}{i}\frac{\partial}{\partial\theta_{m}})$ where $\frac{1}{i}\frac{\partial}{\partial\theta_{j}}$ are generators of $\mathbb{T}^{m}$ action, in Lemma \ref{fgh}, we show: \begin{equation}\label{pbvuytre}  S_{h^N_{BF}}(f)(x)=\frac{1}{B_{h_{BF}^N}(z,z)}f(N^{-1}D_{\theta})B_{h^N_{BF}}(e^{i \theta}z,z)|_{\theta=0,\,\,z=\mu^{-1}_{h_{BF}}(x)}
\end{equation}
where $B_{h_{BF}^N}$ is the Bergman kernel for the  Bargmann-Fock space, $\mu_{h_{BF}}=(|z_1|^2,\ldots, |z_m|^2)$ is the moment map with respect to $\T^m$ action and the hermitian metric $h_{BF}=e^{-\|z\|^2}$, where $\|z\|^2=|z_1|^2+\cdots+|z_m|^2$. And for any $f\in C_0^{\infty}(\R^m)$, $f(N^{-1}D_{\theta})$ is defined by the spectral theorem: \begin{equation} \label{laghdhs}f(N^{-1}D_{\theta}) e^{i \langle \xi, \theta \rangle} = f(\frac{\xi}{N}) e^{i \langle \theta, \xi \rangle}, \end{equation}
In Lemma \ref{TUY}, $S_{h^N_{BF}}(f)$ can also be expressed in term of the symplectic potential as   \begin{equation}\label{symp}S_{h_{BF}^N}(f)(x)  =  \frac{1}{B_{h_{BF}^N}(z,z)}\sum_{\alpha\in \Z^m_+}f(\frac{\alpha}{N})\frac{e^{N(u_{BF}(x)+\langle\frac{\alpha}{N}-x,\nabla u_{BF}(x)\rangle)}}{\|z^{\alpha}\|^{2}_{h_{BF}^{N}}}|_{z=\mu_{h_{BF}} ^{-1}(x)}\end{equation} where $u_{BF}(x)$ is the symplectic potential for the Bargmann-Fock space.

Both formulas (\ref{pbvuytre}) and (\ref{symp}) can be generalized to any smooth polarized noncompact complete \kahler toric manifold $(M,\omega)$, where $\omega \in H^{(1,1)}(M,\Z)$ is a $\T^m$ invariant \kahler form. Let $(L,h)$ be a polarization of $M$, i.e., a positive hermitian holomorphic line bundle such that the curvature satisfies $R(h)=\omega$ (see subsection \ref{toric manifold}). Let $\mu_h(z)$ be the moment map associated to $(L,h) \rightarrow (M,\omega)$, let $P$ be the image of the moment map $\mu_{h}(z)$. 
 In section \ref{djghsfdf}, we define the generalized Szasz analytic functions $S_{h^N}(f)$ for $(L,h)$ by using the Bergman kernel and the symplectic potential (see Definition \ref{hgfd}).

 \subsection{Main results}Our first result is that the generalized Szasz analytic function admits a complete asymptotics and each coefficient can be derived from the complete asymptotics of the Bergman kernel.
\begin{maintheo} \label{BBa} Let $(L, h) \rightarrow (M, \omega)$ be a positive hermitian holomorphic
line bundle over a noncompact complete \kahler toric manifold with a proper moment map $\mu_h(z)$. Assume the hermitian metric $h$ and \kahler form $\omega$ are toric invariant and $R(h)=\omega$.  Let $f\in C^{\infty}_0(\R^m)$ and denote $K$ as the compact support of $f$. Let $S_{h^N}(f)(x)$ be the generalized Szasz analytic function in the sense of Definition \ref{hgfd}. If there exists $\gamma>0$ such that $R(h)>- \gamma  Ric(\omega)$, then there exists complete asymptotics on the compact subset $\mu_h^{-1}(K)\subset M$:  $$ S_{h^N}(f)(x) = f(x) + \lcal_1 f(x) N^{-1} + \lcal_2 f(x) N^{-2} + \cdots + \lcal_n f(x) N^{-n}
+ O(N^{- n - 1}) $$ for any positive integer $n$ as $N$ goes to infinity, where $\lcal_j$ is
a differential operator of order $2j$ depending only on the curvature
of the metric $h$; the expansion may be differentiated
infinitely many times.
\end{maintheo}
In our next theorem, we study the boundary behavior of generalized Sazsz analytic function. To be more precise, we study the behavior of $S_{h^N}(f)(x)$ when $x$ tends to a corner or some facets of $P$. In order to do this, we focus on the neighborhood with radius $1/N$ around a corner or a point on some facets, and by using the dilation operator, we magnify the neighborhood by the factor $N$ to get a neighborhood independent of $N$, then it leads to a universal scaling formula which is similar to Poisson Limit Law.

We first consider the behavior of $S_{h^N}(f)(x)$ as $x$ tends to a corner of $P$. We pick a vertex $V$ of $P$, i.e., the image of a fixed point of the real torus $\T^m$ action on $M$ under the
moment map. We use affine
transformations to map $V$ to be the origin and all facets that meet at $V$ to the hyperplanes $x_j=0$. 
In these new coordinates, we may write the
symplectic potential near the origin as \cite{A, SoZ}:
\begin{equation}\label{corndgd}u_{\phi}(x)=\sum_{j=1}^mx_j \log x_j +h(x)\end{equation}
where $h(x)$ is a smooth function on $P$. Now consider the operator $D_{\frac{1}{N}} S_{h^N}D_{\frac{1}{N}}^{-1} $ where $D_{\frac{1}{N}}$ is the dilation operator defined by $D_{\frac{1}{N}} f(x) = f(\frac{ x}{N})$.
By Definition \ref{hgfd} of generalized analytic functions, we have\begin{equation}\label{scaling}\begin{array}{lll}(D_{\frac{1}{N}} S_{h^N}D_{\frac{1}{N}}^{-1} )f(x)&=&\frac{1}{B_{h^{N}}(\mu_h^{-1}(\frac{x}{N}),\mu_{h}^{-1}(\frac{x}{N}))}\mathcal{D}_{h^{N}}f(x)\quad{where}
\\ && \\ \mathcal{D}_{h^{N}}f(x) &=&\sum_{\alpha\in NP\cap \Z^m }f(\alpha)\frac{e^{N(u_{\phi}(\frac{x}{N})+\langle\frac{\alpha}{N}-\frac{x}{N},\nabla u_{\phi}(\frac{x}{N})\rangle)}}{\|z^{\alpha}\|^{2}_{h^{N}}}\end{array}\end{equation}where $\mu_h(z)$ is the moment map and $u_{\phi}(x)$ is the symplectic potential associated to the \kahler toric manifold.  
  Then we show that  $S_{h_{BF}^1}(f)$ is the universal scaling limit of $D_{\frac{1}{N}} S_{h^N}D_{\frac{1}{N}}^{-1}$ as $N$ goes to infinity.
 
More generally, we consider the boundary behavior of $S_{h^N}(f)(x)$ as $x$ tends to some facets of $P$. In this wall case, we dilate in the variables transverse to the boundary of $P$: We again use the affine transformation so that the facets that meet at a corner are given by $x_j=0$, write $x=(x',x'')\in \R^m$ where $x'=(x_1,\ldots,x_k) \in \R^k$ and $x''=(x_{k+1},\ldots,x_m) \in \R^{m-k}$, then define the dilation operator $\hat D_{\frac{1}{N}} f(x',x'')=f( \frac{x'}{N},x'')$. 
We show that $\hat D_{\frac{1}{N}} S_{h^N}\hat D_{\frac{1}{N}}^{-1}$ tends to $S_{h_{BF}^1}(f_{x''})(x')$ as $N$ goes to infinity. Here, $f_{x''}(x')$ means we fix the variables $x''$ and consider $f(x',x'')$ as a function of $x'$ and $S_{h_{BF}^1}(f_{x''})(x')$ is the classical Szasz analytic functions with respect to variables $x'$.

\begin{maintheo}\label{ghfhgg}
With all assumptions in Theorem \ref{BBa}, there exist differential operators $b_{j}$  
 such that for $f\in C_0^{\infty}(\R^m)$,
 \begin{equation}\label{corner}(D_{\frac{1}{N}} S_{h^N}D_{\frac{1}{N}}^{-1} )f(x) = S_{h_{BF}^1}(f)(x) +  b_1(f)(x)N^{-1}+\cdots +b_n(f)(x)N^{-n}+O(N^{-n-1})\end{equation}
 for any positive integer $n$.  In particular, $$b_{1}(f)(x)=\frac{1}{2}\sum_{i,j}a_{ij}(x)\left(e^{-\|x\|}\sum_{\alpha \in \Z_+^m}f_{ij}(\alpha)\frac{x^\alpha}{\alpha!}\right)$$ where $f_{ij}=\frac{\partial ^2 f}{\partial x_i\partial x_j}$ and $(a_{ij}(x))_{i,j}$ denotes the matrix $M^2(x)\nabla^2 h(0)$ where h(x) is defined by (\ref{corndgd}) and $M(x)=diag\{x_1,x_2,...,x_m\}$ is the diagonal matrix where $x=(x_1,\cdots, x_m )$.

More generally, in the wall case, there exist differential operators $c_j$ such that
 \begin{equation}\label{wall}(\hat D_{\frac{1}{N}} S_{h^N}\hat D_{\frac{1}{N}}^{-1} )f(x',x'') = S_{h_{BF}^1}(f_{x''})(x') + c_1(f)(x',x'')N^{-1}+\cdots +c_n(f)(x',x'')N^{-n}+ O(N^{-n-1}) \end{equation}  
 \end{maintheo}
Just as Bernstein's polynomials can be expressed in terms of binomial random variables, the classical Szsaz analytic function can be expressed in terms of Poisson random variables (see section \ref{example}). Theorem \ref{ghfhgg} also has a probalilistic interpretation: if $M$ is $\mathbb{CP}^{1}$ with Fubini-Study metric, then $S_{h_{FS}^N}(f)(x)$ becomes the classical Bernstein polynomial \cite{Z1}: $$S_{h_{FS}^N}(f)(x)=\sum_{k=0}^N {N \choose k}f(\frac{k}{N})(1-x)^{N-k}x^k$$ Thus, $$(D_{\frac{1}{N}}S_{h_{FS}^N}D_{\frac{1}{N}}^{-1})f(x)=\sum_{k=0}^{N}{N \choose k}f(k )(1-\frac{x}{N})^{N-k}(\frac{x}{N})^{k}$$
Then (\ref{corner}) reads \begin{equation}\label{possionge} \sum_{k=0}^{N}{N \choose k}f(k )(1-\frac{x}{N})^{N-k}(\frac{x}{N})^{k}=\sum_{k=0}^\infty e^{-x}\frac{x^k}{k!}f(k)+b_1(f)(x)N^{-1}+\cdots . \end{equation} 
Now we compute the second term $b_1(f)(x)$: On the \kahler toric manifold $(\CP^1, \omega_{FS})$, the
symplectic potential is,
$$u_{FS}(x) = x \log x + (1 - x) \log(1 - x)$$ thus 
$h(x)$ is just $(1 - x) \log(1 - x)$,
hence,
$$b_1(f)(x) =\frac{1}{2}x^2\left(e^{-x}\sum_{\alpha =0}^{\infty}f''(\alpha)\frac{x^\alpha}{\alpha!}\right)$$
Equation (\ref{possionge}) refines the Poisson Limit Law \cite{R} which says that $${N \choose k} (1-\frac{x}{N})^{N - k}
(\frac{x}{N})^k \rightarrow e^{-x} \frac{x^k}{k!} \,\,\, \mbox{as}\,\,\, N\rightarrow \infty $$
Thus our result that $D_{\frac{1}{N}} S_{h^N}D_{\frac{1}{N}}^{-1}$ admits complete asymptotics generalizes the Poisson Limit Law to any \kahler toric manifold.

In section \ref{example}, we will further compute the generalized Szasz analytic function for the unit
ball with the Bergman metric. And we will give some probabilistic interpretation of some
generalized Szasz analytic functions: we will relate $(\CP^1, \omega_{FS})$ with binomial distribution,
the Bargmann-Fock space with Poisson distribution and the Bergman metric on the unit
disk with Pascal distribution.

\bigskip
\emph{Notations}: Throughout the article, denote $|z|^2$ as the vector $(|z_1|^2, \ldots, |z_m|^2)$, then $\phi(|z|^2)$ is denoted as a function of $|z_1|^2, \ldots, |z_m|^2$ and denote $z^{\alpha}=z_1^{\alpha_1}\cdots z_m^{\alpha_m}$.

\bigskip
\textbf{Acknowledgements}: I am sincerely grateful to Prof. S. Zelditch for his patience to guide me to this beautiful math world. I also want to thank the referee for many helpful comments in the original version. This paper will never come out without their helps and many suggestions.
\section{Background}\label{fhdshd}
\subsection{\kahler toric manifolds}\label{toric manifold}
In this section, we discuss some basic properties of complete \kahler toric manifolds, see \cite{A, F, G, KL, Z1} for more details.

\begin{defin}\label{kjhhgggf}A complex $m$-dimensional \kahler toric manifold is a smooth complete \kahler manifold  $(M,\omega)$ equipped with an effective Hamiltonian holomorphic real torus $\mathbb{T}^m=\R^m/2\pi\Z^m $ action and with a corresponding proper moment map $\mu : M \rightarrow \mathbb{R}^m$.
 \end{defin}
 
For example, any bounded pseudoconvex Reinhardt Domain with smooth boundary equipped with its Bergman metric is a noncompact complete \kahler toric manifold. An open set $X \subset \C^m$ is a Reinhardt domain if
$(z_1, \cdots, z_m) \in X$ implies that $(e^{i\theta_1}z_1, \cdots, e^{i\theta_m}z_m) \in X$, i.e., it's closed under the real torus
$\T^m$ action. On any bounded pseudoconvex Reinhardt Domain $X$, there exists a complete \kahler metric: the Bergman metric $\phi_B$ \cite{DO}. Furthermore, $\phi_B$ is invariant
under $\T^m$ action, hence $\T^m$ acts on $X$ in a Hamiltonian way with respect to the Bergman metric, which implies that $(X, \frac{i}{2}\partial\bar\partial \phi_B)$ is a complete \kahler toric manifold.  As an example, we will consider the case of the unit ball in section \ref{example}.

\bigskip 
Let $(M, \omega)$ be a complex $m$-dimensional complete \kahler toric manifold, then there exists an open dense obit $M^o \cong \R^m \times i\T^m$. On the open dense obit, we have the following local logarithmic coordinate \cite{A}:
  $$z=(z_1,\ldots, z_m)=(e^{\frac{\rho_1}{2}+i\theta_1},\ldots , e^{\frac{\rho_m}{2}+i\theta_m})$$ where $\rho=(\rho_1,\ldots , \rho_m) \in \R^m$ and $\theta=(\theta_1,\ldots, \theta_m) \in [0,2\pi]^m$. Thus we can denote $$e^{i\theta}=(e^{i\theta_1},\ldots,e^{i\theta_m})$$ as the $\T^m$ action on the open orbit of $M$ and denote $$D_{\theta}=(\frac{1}{i}\frac{\partial}{\partial\theta_{1}},\ldots,\frac{1}{i}\frac{\partial}{\partial\theta_{m}})$$ where $\{\frac{1}{i}\frac{\partial}{\partial\theta_{k}},k=1,\ldots, m \}$ are generators of $\mathbb{T}^{m}$ action. We denote $$e^{\rho}=(e^{\rho_1},...,e^{\rho_m})=(|z_1|^2, \ldots, |z_m|^2)$$ Now assume $\omega$ is an integral $(1,1)$ form and assume $\omega$ is $\T^m$ invariant. Then there is a positive hermitian holomorphic line bundle $L \rightarrow M$ and an invariant hermitian metric $h$ on $L$ such that the curvature of $h$ satisfies $R(h)=\omega$. Here, the curvature of a hermitian metric $h$ is defined by: \begin{equation}\label{ddsfgds}R(h)=-\frac{i}{2}\partial\bar \partial \log \|e_L\|^2_h\end{equation} where $e_L$ denotes a local holomorphic frame of $L$ over an open set $U \subset M$, and $\|e_L\|_h=h(e_L,e_L)^{\frac{1}{2}}$ denotes the $h$-norm of $e_L$.

The \kahler form is locally given by $\omega=\frac{i}{2}\partial \bar\partial  \phi(z)$, where the associated hermitian metric has the form $h=e^{-\frac{\phi}{2}}$. Since $\omega$ is invariant under the real $\mathbb{T}^m$ action, then the \kahler potential $\phi$ must be of the form $$\phi(z)=\phi(|z|^2)=\phi(e^{\rho})$$ hence, $$\omega=\frac{i}{2}\sum_{j,k}\frac{\partial^2\phi}{\partial\rho_k \partial \rho_j}dz_j\wedge d\bar z_k$$ By a slight abuse of notation, we denote the \kahler potential in the
logarithmic coordinates by $\phi(\rho)$. For example, in the Fubini-Study case, the \kahler potential is $\phi(z)=\phi(|z|^2)=\log (1+|z_1|^2+\cdots+|z_m|^2)$ and we denote $\phi(\rho)=\log(1+e^{\rho_1}+\cdots +e^{\rho_m})$. Positivity of $\omega$ implies that $\phi$ is a strictly convex function of $\rho \in \R^m$.

  The real torus $\T^m$ acts on $M$ in a Hamiltonian way with respect to $\omega$, and its moment map $\mu_{h} :M \to \mathbb{R}^m $ is given by:
 \begin{equation}\mu_{h}(z_{1}, . . . , z_{m}) = \nabla _{\rho}\varphi(\rho) \,\label{utye}\end{equation}
 We denote $P$ as the image of $\mu_h(z)$, then $P$ is a noncompact polyhedral set which is called
the moment polyhedral set. $P$ in fact can be defined by a set of inequalities of
$$\langle x, v_r\rangle \geq \lambda _r, \,\,\, r = 1, \ldots, d$$
where $v_r$ is a primitive element of the lattice and inward-pointing normal to the $r$-th $(n-1)$-
dimensional face of $P$.

The symplectic potential $u_{\phi}$ defined on $P$ associated to the \kahler potential is the Legendre
dual of $\phi(\rho)$ with respect to variables $\rho$, defined as follows: for any $x = (x_1, \ldots, x_m) \in P$,
there is a unique $\rho \in \R^m$ such that $x = \nabla_\rho \phi$. Then the Legendre dual is defined to be the
convex function
\begin{equation}u_\phi(x_1, \ldots, x_m) = \langle x, \rho \rangle -\phi(\rho) \end{equation}
on $P$. By Legendre duality, we have the following pair,
\begin{equation}\label{dualone}x=\nabla_\rho \phi(\rho), \,\,\, \rho=\nabla _x u_\phi(x)\end{equation}
There exists a canonical \kahler metric and symplectic potential, defined as follows: Let
$\ell_r : \R^m \rightarrow \R$ be the affine functions,
$$\ell_r(x)=\langle x, v_r \rangle -\lambda_r$$
Then the canonical symplectic potential is defined by
$$u_0(x) =\sum_{r=1}^d \ell_r (x)\log \ell_r(x) $$
which in turn corresponds to a canoncial \kahler potential \cite{G}. Every symplectic potential
has the same singularities on the boundary $\partial P$ as the canonical symplectic potential, and
must be in the form of \cite{A}:
\begin{equation}\label{sympotential}u_\phi(x) = u_0(x) + g(x)  \end{equation}
where $g(x)$ is smooth on $P$, and the matrix $G_\phi(x) = \nabla^2
_xu_\phi(x)$ is positive definite on $P^o$
which is the interior of $P$.

We also denote by $H_\phi(\rho) =\nabla^2_\rho \phi(e^\rho)$ the Hessian of the \kahler potential on the open orbit
in $\rho$ coordinates. By Legendre duality again,
\begin{equation}\label{dualtwo}H_\phi(\rho)=  G^{-1}_\phi(x),\,\,\, \mu_h(e^\rho)=x \end{equation}

 \subsection{Bergman kernel }\
In this section, we first review the definition of the Bergman kernel for any polarized \kahler manifold. Then we sketch the proof of the existence of the complete asymptotics of the Bergman kernel on the diagonal.
 \subsubsection{Definition}
Let $( M, \omega)$ be a complete \kahler manifold of dimension $m$ (over $\C$). Let $(L,h)\rightarrow M$ be a positive hermitian holomorphic line bundle. Assume there is a hermitian metric $h$ on $L$ such that the curvature $R(h)=\omega$.  $h$ induces a hermitian metric $h^{N}$ on $L^{N}$ by $\|{e_{L}^{\otimes N}}\|_{h^{N}}=\|{e_{L}}\|_{h}^{N}$, where $e_L$ is a local holomorphic frame. Locally we can write $\omega=\frac{i}{2} \partial \bar \partial \phi$, thus $h= e^{-\frac{\phi}{2}}$ and $h^N=e^{-\frac{N\phi}{2}}$.

Let $H^{0}_{L^2}(M,L^{N})$ be the weighted space of all $L^{2}$ global holomorphic sections of the line bundle $L^{N}$. This means for each local frame $e_{L}$ of $L$ , we can write each section of $H^{0}_{L^2}(M,L^{N})$ as $s=f e^{\otimes N}_{L}$, where $f$ is a holomorphic function and satisfies:
$$\|s\|^2_{h^N}=\int _M |f|^2 e^{-N \phi} \frac{\omega ^{m}}{m!}< \infty $$ Thus $H^{0}_{L^2}(M,L^{N})$ will be a Hilbert space with the natural inner product $Hilb_N(h)$ induced by  $h^{N}$: \begin{equation} \label{gfsgsh} \langle s_{1},s_{2}\rangle _{Hilb_N(h)}= \int_{M} h^{N}(s_{1},s_{2})\frac{\omega ^{m}}{m!}=\int _M f_1\bar f_2 e^{-N \phi} \frac{\omega ^{m}}{m!}
 \end{equation} where $s_{1}=f_1e^{\otimes N}_L,s_{2}=f_2e^{\otimes N}_L\in H^{0}_{L^2}(M,L^{N})$.

We define the Bergman kernel as the orthogonal projection from the $L^2$ sections to the holomorphic $L^2$ sections with respect to the inner product $Hilb_N(h)$,
\begin{equation}B_{h^N}: L^{2}(M,L^{N})\rightarrow H^{0}_{L^2}(M,L^{N})\end{equation}
Furthermore, if $\{s_{j}\}_{j=1} ^{}$ is an orthonormal basis of  $H^{0}_{L^2}(M,L^{N})$, then: \begin{equation}B_{h^N}(z,w)=\sum_{j=1}^{}s_j(z)\otimes \overline{s_j(w)} \label{ldh}\end{equation}
Note that in the compact case, the dimension of $H^0_{L^{2}}(M, L^N)$ is finite, but in the noncompact case, it is generally infinite.
\subsubsection{Bergman kernel for toric manifolds}
Let $(M,\omega)$ be a polarized complete \kahler toric manifold and assume $\omega$ is $\T^m$ invariant. Let $(L,h)$ be a positive hermitian holomorphic line bundle over $M$ such that $R(h)=\omega$, then we have the following proposition \cite{F,G}:
\begin{prop}\label{monim}The global holomorphic sections of $L^N$ are $$ H^{0}(M,L^{N}) = \bigoplus_{\alpha \in NP\cap \mathbb{Z}^m } \mathbb{C}  \cdot z^{\alpha}$$ 
\end{prop} For convenience, throughout the article, we assume that for each $\alpha\in \mathbb{Z}^m \cap NP $, \begin{equation}\label{assume} \|z^{\alpha}\|^2_{h^N} =\int_M z^{\alpha}\bar z^{\alpha}e^{-N\phi} \frac{\omega^m}{m!}<\infty\end{equation} 
Under this assumption, we have $$ H^{0}_{L^2}(M,L^{N}) = \bigoplus_{\alpha \in NP\cap \mathbb{Z}^m } \mathbb{C}  \cdot z^{\alpha}$$ But this assumption is not necessary, without this condition, some $z^\alpha$ may be not in the space of holomorphic $L^2$ sections, but our results and arguments are still valid as long as we have complete asymptotics of Bergman kernel. 

Now claim that all such  monomials in fact form an orthogonal basis of $H^{0}_{L^2}(M,L^{N})$ with respect to the inner product $Hilb_N(h)$. This can be seen if we integrate on the dense open obit $M^o \cong \R^m \times i\T^m$: $$\begin{array}{lll} \langle z^{\alpha}, z^{\beta}\rangle_{Hilb_N(h)} &=& \int_M z^{\alpha}\bar z^{\beta}e^{-N\phi} \frac{\omega^m}{m!}  =\int_{\R^m \times i\T^m} z^{\alpha}\bar z^{\beta}e^{-N\phi} \frac{\omega^m}{m!}  \\ && \\&= &\left (\int_{[0,2\pi]^m} e^{i\langle \alpha-\beta,\theta \rangle}d\theta\right)  \left( \int _{\R^m}e^{\langle \frac{\rho}{2}, \alpha+\beta \rangle }e^{-N\phi(\rho)} \det H_{\rho}(\phi)d\rho \right) \\ && \\&= & \delta _{\alpha,\beta} \int _{\R^m}e^{\langle \frac{\rho}{2}, \alpha+\beta \rangle }e^{-N\phi(\rho)} \det H_{\rho}(\phi)d\rho\end{array}$$
  where $H_{\rho}(\phi)$ is the Hessian of $\phi(\rho)$ with respect to $\rho \in \R^m$. Thus by identity (\ref{ldh}), the Bergman Kernel off the diagonal is
\begin{equation}\label{ddsdgdsg}B_{h^N}(z,w)=\sum _{\alpha\in NP\cap \mathbb{Z}^m }\frac{z^{\alpha}\bar w^{\alpha}e^{-\frac{N\phi(|z|^2)}{2}-\frac{N\phi(|w|^2)}{2}}}{\|z^{\alpha}\|_{h^{N}}^{2}}\end{equation}

\subsubsection{Complete asymptotics }
  Let $(L,h)$ be a positive hermitian holomorphic line bundle over the compact \kahler manifold $(M,\omega)$. Assume $R(h)=\omega$, then the Bergman kernel on the diagonal admits complete asymptotics which is the well known Tian-Yau-Zelditch Theorem \cite{T,Z2}, a very nice proof is given in \cite{BBS}. For the noncompact case, the similar result hold:
 \begin{theo}  \label{A}
Let $(M,\omega)$ be a complex $m$-dimensional noncompact complete \kahler manifold, let $(L,h)\to (M,\omega)$ be a positive hermitian holomorphic line bundle such that $R(h)=\omega$. Assume there exists $\gamma>0$, such that $R(h)>-\gamma Ric(\omega)$, then on any compact subset $K \subset M$, for the Bergman kernel on the diagonal, we have:
\begin{equation}B_{h^N}(z,z)=\sum_{j}\|s_{j}(z)\|_{h^{N}}^{2}=N^{m}(1+a_{1}(z)N^{-1}+a_{2}(z)N^{-2}+\cdots ) \label{CX}\end{equation}
where $a_{n}(z)$ are smooth on $ K$.

In particular, for the complete \kahler toric manifold with the \kahler potential $\phi(|z|^{2})$, on any compact subset, we have: \begin{equation}B_{h^N}(z,w)= e^{N(\phi(z \cdot \bar{w})-\frac{1}{2 }
(\phi(|z|^{2})+\phi(|w|^{2})))}A_{N}(z,w) \,\, mod N^{-\infty}
\label{FUNN}\end{equation}
where $\phi(z\cdot \bar{w })$ is the almost analytic extension of $\phi(|z|^{2})$ and $A_{N}(z,w)= N^{m}(1+a_{1}(z,w)N^{-1}+\cdots)$ a semi-classical symbol of order $m$.
\end{theo}

\emph{Sketch of the proof}:  The proof follows the one in \cite{BBS}. Locally, we fix our small coordinate neighborhood to be the unit ball $B^m$ of $\C^m$, $u$ is a holomorphic function on $B^m$, we can define the local expression of the norm of a section of $L^N$ over $B^m$ as: $$\|u\|^2_{h^N, loc}=\int_{B^m}|u|^2e^{-N\phi}\frac{\omega^m}{m!}$$In Proposition 2.7 of \cite{BBS}, it's proved that: for any positive integer $k$, there exists a local asymptotic Bergman kernel: \begin{equation}B_{h^N,loc}^k=N^m(1 + a_{1}(z,\bar w)N^{-1}+...+ a_{k}(z,\bar w)N^{-k}) \label{cxz}\end{equation} where $a_{j}(z,\bar w)$ are defined in a fixed neighborhood of $z$ explicitly given such that:\begin{equation}u(x)=\langle \chi u, B_{h^N,loc}^k \rangle_{h^N,loc}+O(N^{-k-1})\|u\|_{h^N,loc}\label{cxzfgh}\end{equation} where $\chi$ is a smooth function supported in the unit ball $B^m$ and equal to one on the ball of radius 1/2. This estimate is a local property, it will be automatically true for the noncompact \kahler manifold.

Next, globally, they show that there exists a uniform $\delta>0$, whenever $d(z,w)<\delta$, we have:
\begin{equation}B_{h^N}(z,w)=B_{h^N,loc}^{k}(z,w)+O(N^{m-k-1})\label{ct}\end{equation} This estimate imply that the difference between the local Bergman kernel and the global one is up to a small error term. Thus we can get the complete asymptotics of $B_{h^N}(z,w)$ by the complete asymptotics of $B_{h^N,loc}^{k}(z,w)$.

In our noncompact case, if we can prove (\ref{ct}) for any fixed compact subset $K$, then Theorem \ref{A} follows. In order to prove (\ref{ct}), first choose $\chi$ as the cut-off function equal to $1$ in a neighborhood of $z$ which is large enough to contain $w$. In the estimate (\ref{cxzfgh}), let's choose $u=B_{h^N}(z,w)$, then we have $$B_{h^N}(z,w)=\langle\chi B_{h^N}, B_{h^N,loc}^k \rangle_{h^ N}+O(N^{-k-1}\|B_{h^N}\|) \label{cvbnmt}$$ By Theorem 2.1 in \cite{Be}, there exist a constant $C_K$ which depends on $K$ such that we have upper bound for the Bergman kernel $$ B_{h^N}(z,w)\leq C_KN^m$$ then
\begin{equation}B_{h^N}(z,w)=\langle\chi B_{h^N}, B_{h^N,loc}^k \rangle_{h^ N}+O(N^{m-k-1})\label{cvbnmt}\end{equation} Next let's estimate \begin{equation}v_{N}(z,w)=\chi B_{h^N,loc}^{k}(z,w)- \langle \chi B_{h^N,loc}^{k},B_{h^N}\rangle_{h^N}\end{equation}  Since the inner product is the Bergman projection, $v_{N}(z,w)$ will be the
$L^2$-minimal solution to the $\bar \partial$-equation $$\bar \partial v_{N}= \bar \partial (\chi B_{h^N,loc}^{k})$$ Because of the curvature assumption $R(h)>-\gamma Ric(\omega)$, we can apply the H\"{o}mander $L^2$-$\bar{\partial}$ estimate (Theorem 5.1 in \cite{D}). Then we have the square norm estimate $$\|v_{N}(z,w)\|_{h^N}^2 \leq O(k^{-j})$$ for any positive integer $j$. Then by Cauchy integral formula in a ball around $z$ of radius $1/N^{1/2}$, we have the pointwise estimate $$|v_{N}(z,w)|^2 \leq O(k^{-j})$$ Thus $$B_{h^N}=B_{h^N,loc}^k -v_{N}+O(N^{m-k-1})=B_{h^N,loc}^k +O(N^{m-k-1})$$Then the theorem follows the complete asymptotics of $B_{h^N,loc}^k$.

Finally, refer to \cite{SZ} for the proof of the last statement about the Bergman kernel off diagonal over the \kahler toric manifold.

\begin{remark} Note that, in the proof we have to restrict the Bergman kernel to any compact subset to ensure the existence of the uniform constant $C_K$ and we also need the curvature assumption to ensure H\"{o}mander $L^2$-$\bar{\partial}$ estimate.
\end{remark}

\section{\label{F} Bargmann-Fock space and Classical Szasz Analytic Functions }
In this section, we first discuss some basic properties of the Bargmann-Fock space, then relate the classical Szasz analytic function with it.
\subsection{ Bargmann-Fock space}\label{bf} The Bargmann-Fock  space is the space of entire functions on $\mathbb{C}^{m}$ which are $L^{2}$ integral with respect to the Bargmann-Fock metric, i.e., $\mathcal{H}^{2}(\mathbb{C}^{m},\pi^{-m} e^{-N\|z\|^{2}}dz d\bar{z})$.
This space is a  Hilbert space and the orthonormal basis is given by monomials \cite{SZ}:
$$\left\{\frac{z^{\alpha}}{\sqrt{\frac{\alpha!}{N^{m+|\alpha|}}}},\alpha \in \mathbb{Z}_{+}^{m}\right\}$$ where we denote $|\alpha|=|\alpha_1|+\cdots+|\alpha_m|$. Then it's easy to get the Bergman kernel on the diagonal for the Bargmann-Fock space
$$\begin{array}{lll} B_{h_{BF}^N}(z,z)&=&\sum _{\alpha \in \Z^m_+}\frac{z^{\alpha}\bar z^{\alpha}}{\|z^{\alpha}\|^{2}_{h_{BF}^{N}}} e^{ -N\|z\|^{2}}= \sum _{\alpha\in \Z^m_+}\frac{z^{\alpha}\bar z^{\alpha}}{\frac{\alpha!}{N^{m+|\alpha|}}} e^{ -N\|z\|^{2}}=N^{m} \label{UIO}\end{array}$$
In the Bargmann-Fock space, the K\"{a}hler  potential for the \kahler manifold $\C^m$ is \begin{equation}\varphi(|z|^2) =\|z\|^{2}=|z_1|^2+\cdots+|z_m|^2 \end{equation} i.e., $\phi(\rho)= e^{\rho_1}+\cdots + e^{\rho_m}$.
Thus the moment map is
\begin{equation}\mu_{h_{BF}} (z_{1}, \ldots, z_{m})= \nabla_{\rho}\phi(\rho)= (|z_{1}|^{2}, \ldots, |z_{m}|^{2})\label{RTYE}\end{equation}
In view of this moment map, the moment polyhedral set for the \kahler toric manifold $\C^m$ with the Euclidean metric is the orthant  $\mathbb{R}^{m}_{+}$.

The symplectic potential defined on $\mathbb{R}^{m}_{+}$ which is the Legendre dual of $\varphi(\rho)$ with respect to the variable $\rho$ is given by\begin{equation}u_{BF}(x)=\sum_{j=1}^{m} x_{j}\log x_{j}-\sum_{j=1}^{m} x_{j}\label{NM}\end{equation}where $x=(x_1,\ldots, x_m) \in \mathbb{R}^{m}_{+}$.
\subsection{Two relations} In this subsection, we get two relations between the Bargmann-Fock space and the classical Szasz analytic function.

First if we apply the pseudo-differential operator $f(N^{-1}D_{\theta})$ to the Bergman kernel off diagonal for the Bargmann-Fock space, we have
\begin{lem}\label{fgh}Assume $f$ is a smooth Schwartz function on $\R^m_+$, then the classical Sazsz analytic function can be expressed in term of the Bergman kernel as:
\begin{equation}S_{h_{BF}^N}(f)(x)=\frac{1}{B_{h_{BF}^N}(z,z)}f(N^{-1}D_{\theta})B_{h^N_{BF}}(e^{i \theta}z,z)|_{\theta=0,\,\,z=\mu^{-1}_{h_{BF}}(x)}\label{poiuytre}
\end{equation}
\end{lem}
\begin{proof}The key is to rewrite the Bergman kernel off the diagonal as
\begin{equation}\label{poi}\begin{array}{lll} B_{h_{BF}^N}(e^{i\theta}z,z)&=&\sum _{\alpha\in \mathbb{Z}_+^m}e^{i\langle\alpha,\theta\rangle}\frac{z^{\alpha}\bar z^{\alpha}}{\|z^{\alpha}\|^{2}_{h_{BF}^{N}}} e^{ -N\|z\|^{2}}=e^{ -N\|z\|^{2}}\sum _{\alpha \in \mathbb{Z}_+^m}e^{i\langle\alpha,\theta\rangle}\frac{z^{\alpha}\bar z^{\alpha}}{\frac{\alpha!}{N^{m+|\alpha|}}}  \\ && \\
& = &N^{m}e^{-N\|x\|}\sum _{\alpha \in \mathbb{Z}_+^m}\frac{(Nx)^{\alpha}}{\alpha!}e^{i\langle\theta,\alpha\rangle},\,\, \mbox{when}\, x=\mu_{h_{BF}}(z)\end{array}\end{equation}
About the operator $f(N^{-1}D_{\theta})$ which is defined by identity (\ref{laghdhs}), we have the following general formula:
\begin{equation} \label{general}f(N^{-1}D_{\theta}) \psi(e^{i\theta}w)|_{\theta=0} = \int_{\R^{m}}\hat{f}(\xi)\psi(e^{i\frac{\xi}{N}}w)d\xi \end{equation}for any smooth function $\psi$.

Now we apply the operator $f(N^{-1}D_{\theta})$ on both sides of (\ref{poi}) and use the fact $B_{h_{BF}^N}(z,z)=N^{m}$, then we complete the proof by
$$ \begin{array}{lll} && \frac{1}{B_{h_{BF}^N}(z,z)}f(N^{-1}D_{\theta})B_{h^N_{BF}}(e^{i \theta}z,z)|_{\theta=0,\,\,z=\mu^{-1}_{h_{BF}}(x)} \\ && \\
& = & e^{-N\|x\|}\sum _{\alpha \in \mathbb{Z}_+^m}\frac{(Nx)^{\alpha}}{\alpha!}\int_{\R^m}\hat f(\xi) e^{i\langle \xi,N^{-1}\alpha\rangle}d\xi= e^{-N\|x\|}\sum _{\alpha \in \mathbb{Z}_+^m}f(\frac{\alpha}{N})\frac{(Nx)^{\alpha}}{\alpha!} \end{array}$$
\end{proof}
 We can also get another relation:
\begin{lem}\label{TUY}The classical Szasz analytic function can also be expressed by the symplectic potential of the Bargmann-Fock space as
\begin{equation} \begin{array}{lll}S_{h_{BF}^N}(f)(x) & = & \frac{1}{B_{h_{BF}^N}(z,z)}\sum_{\alpha\in \Z^m_+}f(\frac{\alpha}{N})\frac{e^{N(u_{BF}(x)+\langle\frac{\alpha}{N}-x,\nabla u_{BF}(x)\rangle)}}{\|z^{\alpha}\|^{2}_{h_{BF}^{N}}}|_{z=\mu_{h_{BF}} ^{-1}(x)}
\end{array}  \end{equation} where $\mu_{h_{BF}} (z)$ is the moment map, $B_{h_{BF}^N}(z,z)$ is the Bergman kernel for the Bargmann-Fock space, $u_{BF}(x)$ is the symplectic potential over the orthant $\mathbb{R}^{m}_{+}$.
\end{lem}
\begin{proof}A simple calculation shows that the Szasz terms may be expressed
in terms of the symplectic potential as $$\frac{ e^{-N\|x\|}(Nx)^{\alpha}}{\frac {\alpha !}{N^{m+|\alpha|}}}=\frac{e^{N(u_{BF}(x)+\langle\frac{\alpha}{N}-x,\nabla u_{BF}(x)\rangle)}}{\|z^{\alpha}\|^{2}_{h_{BF}^{N}}}\,\,\, ,z=\mu_{h_{BF}} ^{-1}(x)$$Thus
$$ \begin{array}{lll}S_{h_{BF}^N}(f)&=& e^{-N\|x\|}\sum _{\alpha \in \mathbb{Z}_+^m}\frac{(Nx)^{\alpha}}{\alpha!}f(\frac{\alpha}{N})= N^{-m}\sum _{\alpha \in \mathbb{Z}_+^m}f(\frac{\alpha}{N})\frac{ e^{-N\|x\|}(Nx)^{\alpha}}{\frac {\alpha !}{N^{m+|\alpha|}}} \\ && \\
& = &\frac{1}{B_{h_{BF}^N}(z,z)}\sum_{\alpha \in \Z^m_+}f(\frac{\alpha}{N})\frac{e^{N(u_{BF}(x)+\langle\frac{\alpha}{N}-x,\nabla u_{BF}(x)\rangle)}}{\|z^{\alpha}\|^{2}_{h_{BF}^{N}}}
\end{array}  $$\end{proof}

\section{\label{djghsfdf} Generalized Szasz Analytic functions}

\subsection{Definition} We now generalize the two relations we get in the last section to any polarized \kahler toric manifold to define the generalized Szasz analytic function. Following notations and assumptions in section \ref{fhdshd}, we have

\begin{defin}\label{hgfd}Let $f \in C_{0}^{\infty}(\R^m)$, we define the generalized Szasz analytic function of $f(x)$ by the symplectic potential as:
$$\begin{array}{lll}\label{djgh}S_{h^N}(f)(x)&=&\frac{1}{B_{h^{N}}(z,z)}\mathcal{N}_{h^{N}}f(x) \,\,\mbox{where}
\\ && \\ \mathcal{N}_{h^{N}}f(x) &=&\sum_{\alpha\in \Z^m \cap NP}f(\frac{\alpha}{N})\frac{e^{N(u_{\phi}(x)+\langle\frac{\alpha}{N}-x,\nabla u_{\phi}(x)\rangle)}}{\|z^{\alpha}\|^{2}_{h^{N}}} \label{dddss}
\end{array}$$ 
where $z$ and $x$ are related by the moment map $z=\mu_{h}^{-1}(x)$.\end{defin}

The following proposition says that this definition also generalizes Lemma \ref{fgh} to any noncompact toric \kahler manifold:
\begin{prop} \label{SBERG}$S_{h^N}(f)$ is the generalized Szasz analytic functions defined above, then
\begin{equation}S_{h^N}(f)(x)=\frac{1}{B_{h^N}(z,z)}f(N^{-1}D_{\theta})B_{h^N}(e^{i \theta}z,z)|_{\theta=0,\,z=\mu_{h}^{-1}(x)}
\end{equation}
\end{prop}
\begin{proof}
 By identity (\ref{ddsdgdsg}), we have
\begin{equation}\label{off}B_{h^N}(e^{i\theta}z,z)=\sum _{\alpha\in \mathbb{Z}^m \cap NP}\frac{z^{\alpha}\bar z^{\alpha}e^{-N\phi(|z|^2)}e^{i\langle\alpha,\theta\rangle}}{\|z^{\alpha}\|_{h^{N}}^{2}}\end{equation}
Each term in the Bergman kernel can be expressed in term of symplectic potential as
\begin{equation}\label{ddsdgdsdddg}\frac{z^{\alpha}\bar z^{\alpha}e^{-N\phi(|z|^2)}}{\|z^{\alpha}\|_{h^{N}}^{2}}=\frac{e^{N(u_{\phi}(x)+\langle\frac{\alpha}{N}-x,\nabla u_{\phi}(x)\rangle)}}{\|z^{\alpha}\|^{2}_{h^{N}}}\,, \mbox{when} \, \mu_{h}(z)=x\end{equation}
this follows from the pair of identities,
$$z^{\alpha}\bar z^{\alpha}=e^{\langle\alpha, \nabla u_{\phi}(x)\rangle}, \,\, \phi(|z|^2)=\langle x, \nabla u_{\phi}(x)\rangle-u_{\phi}(x)\,, \mbox{when} \, \mu_{h}(z)=x$$
Now apply the general formula (\ref{general}) to (\ref{off}), we have
\begin{equation}\label{dfddvdsdgdsdddg}\begin{array}{lll}f(N^{-1}D_{\theta})B_{h^N}(e^{i \theta}z,z)|_{\theta=0,z=\mu_{h}^{-1}(x)}&=& \left(\int_{\R^{m}}\hat{f}(\xi)B_{h^N}(e^{i \frac{\xi}{N} }z,z)d\xi\right) |_{z=\mu_{h}^{-1}(x)}\\ && \\
& = &
\sum_{\alpha\in \mathbb{Z}^m \cap NP}\left(\frac{z^{\alpha}\bar z^{\alpha}e^{-N\phi(|z|^{2})}}{\|z^{\alpha}\|_{h^{N}}^{2}} \int_{\R^{m}}\hat{f}(\xi)e^{i\langle \frac{\alpha}{N},\xi\rangle}d\xi\right)|_{z=\mu_{h}^{-1}(x)}\\ && \\
& = &\left(\sum _{\alpha\in \mathbb{Z}^m \cap NP} f(\frac{\alpha}{N})\frac{z^{\alpha}\bar z^{\alpha}e^{-N\phi(|z|^{2})}}{\|z^{\alpha}\|_{h^{N}}^{2}}\right)|_{z=\mu_{h}^{-1}(x)}\\ && \\
& = &\mathcal{N}_{h^{N}}f(x) \end{array}
\end{equation} in the last step, we use the identity (\ref{ddsdgdsdddg}). Thus we complete the proof by dividing $B_{h^N}(z,z)$ on both sides.
\end{proof}

\subsection{Complete asymptotics } Let's turn to the proof of Theorem \ref{BBa} which says that $S_{h^N}(f)$ admits complete asymptotics.
\begin{proof}  First claim: there exist differential operators $\mathcal{N}_j$ such that  \begin{equation} \mathcal{N}_{h^{N}}f(x)=N^m f(x)+N^{m-1}\mathcal{N}_1f(x)+\cdots \end{equation} where $\mathcal{N}_j$ are computable from the complete asymptotics of $B_{h^N}(z,z)$.

To prove the claim, we denote $K$ as the compact support of $f(x)$, thus $ \mu_h^{-1}(K)$ will be a compact subset of $M$ since $\mu_h$ is proper. Thus by formula (\ref{FUNN}), on $ \mu_h^{-1}(K)$ we have: $$B_{h^N}(z,w)= e^{N(\phi(z \cdot \bar{w})-\frac{1}{2 }
(\phi(|z|^{2})+\phi(|w|^{2})))}A_{N}(z,w) \,\, mod N^{-\infty} $$ where $\phi(z\cdot\bar w)$ is the holomorphic extension of $\phi(|z|^2)$ and $A_{N}(z,w)$  admits a complete asymptotics.  Let's choose $w=e^{iN^{-1}\xi}z$, then
\begin{equation}\label{dgcvfsag}B_{h^N}(e^{iN^{-1}\xi}z,z)= e^{N(\phi(e^{iN^{-1}\xi}|z|^{2})-
\phi(|z|^{2}))}A_{N}(ze^{iN^{-1}\xi},z)\,\, mod N^{-\infty}\end{equation}
where $e^{iN^{-1}\xi}|z|^{2}$ is the vector $(e^{iN^{-1}\xi_1}|z_1|^{2},\ldots , e^{iN^{-1}\xi_m}|z|^{m})$.

Plug (\ref{dgcvfsag}) into the first line of equation (\ref{dfddvdsdgdsdddg}), we have following crucial expression:
\begin{equation}\label{dgfsag}\begin{array}{lll} \mathcal{N}_{h^{N}}f(x)&=&\int_{\R^{m}}\hat{f}(\xi)e^{i\langle \xi,N^{-1}D_{\theta} \rangle}B_{h^N}(e^{i \theta}z,z)|_{\theta=0, z=\mu_h^{-1}(x)}d\xi \\ && \\
& = & \left(\int_{\R^m}\hat{f}(\xi)B_{h^N}(e^{iN^{-1}\xi}z,z)d\xi \right)|_{z=\mu_h^{-1}(x)}\\ && \\
& = &\left(\int_{\R^{m}}\hat{f}(\xi) e^{N(\phi(e^{iN^{-1}\xi}|z|^{2})-
\phi(|z|^{2}))}A_{N}(e^{iN^{-1}\xi}z,z)d\xi\right)|_{z=\mu_h^{-1}(x)}
\end{array}\end{equation}  To the phase of this integral, following the one in \cite{Z1}, we have
\begin{equation}\label{dgdffsag}\begin{array}{lll}\phi(e^{iN^{-1}\xi} |z|^{2})-\phi(|z|^{2})&=&\int_{0}^{1}\frac{d}{dt}\phi(e^{itN^{-1}\xi } |z|^{2})dt\\ && \\
& = &iN^{-1}\int_{0}^{1}\langle\nabla_{\rho}\phi(e^{itN^{-1}\xi+\rho}),\xi\rangle\ dt\\ && \\
& = &iN^{-1}\langle\nabla_{\rho}\phi(e^{\rho}),\xi\rangle+\frac{1}{2}(iN)^{-2}\int_{0}^{1}(t-1)^{2}\nabla_{\rho}^{2}\phi(e^{itN^{-1}\xi+\rho}){\xi}^{2} dt
\\ && \\
& = &iN^{-1}\langle \mu_h(z),\xi\rangle+\frac{1}{2}(iN)^{-2}\nabla_{\rho}^{2}(\phi(e^{\rho}))\xi^{2}+O(N^{-3})\\ && \\
& = &iN^{-1}\langle\mu_h(z),\xi\rangle-\frac{1}{2}N^{-2}\langle H_{\phi}(\rho)\xi,\xi\rangle +O(N^{-3})
\end{array}\end{equation} where $H_{\phi}(\rho)=\nabla_{\rho}^{2}(\phi(\rho))$ is the Hessian of the \kahler potential with respect to $\rho$ variables. Thus we can rewrite (\ref{dgfsag}) as the following estimate, $$\mathcal{N}_{h^{N}}f(x)= \int_{\R^{m}}\hat{f}(\xi) e^{i\langle\mu_h(z),\xi\rangle}e^{-\frac{1}{2}N^{-1}\langle H_{\phi}(\rho)\xi,\xi\rangle +O(N^{-2})}A_{N}(ze^{iN^{-1}\xi},z)d\xi
$$ Apply the Taylor expansion to the factor $e^{-\frac{1}{2}N^{-1}\langle H_{\phi}(\rho)\xi,\xi\rangle +O(N^{-2})}$, one obtains a new amplitude $\tilde{A}_N$ such that $$\mathcal{N}_{h^{N}}f(x)= \int_{\R^{m}}\hat{f}(\xi) e^{i\langle\mu_h(z),\xi\rangle}\tilde{A}_{N}(ze^{iN^{-1}\xi},z)d\xi
$$ The amplitude $\tilde{A}_N$ has an expansion of the form,
$$\tilde{A}_N(ze^{iN^{-1}\xi},z)=N^m +N^{m-1}a_1+O(N^{m-2})$$
where each coefficient $a_j$ is smooth.

Thus by the inversion formula of the Fourier transformation, we have: $$\mathcal{N}_{h^{N}}f(\mu_h(z))= N^{m}f(\mu_h(z))+N^{m-1}[-\frac{1}{2}\langle H_{\phi}(\rho)D_{x},D_{x}\rangle f(\mu_h(z))+a_{1}(z,z)f(\mu_h(z))]+O(N^{m-2})$$ Thus we get the claim since $\mu_h(z)=x$.

Dividing $\mathcal{N}_{h^{N}}f(x)$ by $B_{h^N}(z,z)$ and using the complete asymptotics of $B_{h^N}(z,z)$ completes the proof of Theorem \ref{BBa}.

\end{proof} \subsection{\label{AA}Scaling asymptotics }
\subsubsection{Proof of Theorem \ref{ghfhgg} }We now turn to the proof of Theorem \ref{ghfhgg}.
\begin{proof}
Now pick up a vertex $V$ of $P$, i.e., the image of a fixed point of $\T^m$ action on $M$ under the
moment map. We use the affine
transformation to map $V$ to the origin and all facets that meet at $V$ to the hyperplanes $x_j = 0$.

Let $\mu_h(\tilde z_N) =\frac{ x}
{N}$ , denote $H_\phi(\tilde \rho _N)$ as the Hessian of the \kahler potential evalued at the
point $\tilde \rho_N$ where $|\tilde z_N|^2= e^{\tilde \rho_N}$.

Since $\frac{ x}
{N}$ is near the origin as $N$ large enough, in these new coordinates, we may write the
symplectic potential (\ref{sympotential}) near the origin as (p.42 in \cite{SoZ}):
\begin{equation}u_{\phi}(x)=\sum_{j=1}^mx_j \log x_j +h(x)\label{hx}\end{equation}
where $h(x)$ is a smooth function on $P$.

First, the Hessian of the symplectic potential can be decomposed into the sum, $$G_\phi(x)=\sum_{j=1}^m\frac{1}
{x_j}\delta_{jj}+\nabla^2 h:= M(\frac{1}{x})+\nabla^2 h(x)$$Thus by formula (\ref{dualtwo}) of Legendre duality,
$$H_\phi(\tilde \rho_N)=G_\phi^{-1}(\frac{x}{N})=(M(\frac{N}{x})+\nabla^2 h(\frac{x}{N}))^{-1}$$
Thus, $$H_\phi(\tilde \rho_N)=M(\frac{x}{N})-M(\frac{x}{N})^2\nabla^2 h(\frac{x}{N})+\cdots$$
 Apply the Taylor expansion to the matrix $\nabla^2 h(\frac{x}{N})$ at $0$, we have the complete asymptotics
\begin{equation}\label{matrixexpansion}H_\phi(\tilde \rho_N)=M(\frac{x}{N})-M(\frac{x}{N})^2\nabla^2 h(0)+\cdots\end{equation}
 Second, by formula (\ref{dualone}) of Legendre duality,
 $$\tilde \rho_N=\nabla_x u_\phi(\frac{x}{N})=(\log \frac{x_1}{N}, \ldots, \log \frac{x_m}{N})+\nabla h(\frac{x}{N})+\stackrel{\rightarrow}{1}$$
 Apply Taylor expansion to the vector $\nabla h(\frac{x}{N})$ at $0$, we have
 $$|\tilde z_N|^2=e^{\tilde \rho_N}=( \frac{x_1}{N}e^{1+h_{x_1}(0)+h_{x_1x_1}(0) \frac{x_1}{N}+\cdots},\ldots,\frac{x_m}{N}e^{1+h_{x_m}(0)+h_{x_mx_m}(0) \frac{x_m}{N}+\cdots} )$$
 which implies that $|\tilde z_N|^2$ admits complete expansion by Taylor expansion. For example, for
the $j$-th element, we have
 
 \begin{equation}\label{zexpansion} |\tilde z_{N,j}|^2= (1+h_{x_j}(0))\frac{x_j}{N}+h_{x_jx_j}(0)(\frac{x_j}{N})^2+\cdots\end{equation}
Now denote $f_{N}(x)=D_{\frac{1}{N}}^{-1}f(x)$, then apply identity (\ref{dgfsag}), we have:
$$\begin{array}{lll}&&(D_{\frac{1}{N}} S_{h^N}D^{-1}_{\frac{1}{N}})f(x)=D_{\frac{1}{N}} S_{h^N}f_N(x)\\ &&\\ &=&
 D_{\frac{1}{N}}\left( \int_{\R^{m}}\widehat{f_N}(\xi) e^{N(\phi(e^{i\frac{\xi}{N}}|z|^{2})-
\phi(|z|^{2}))}A_{N}(z,ze^{i\frac{\xi}{N}})/B_{h^N}(z, z)d\xi\right), \mbox{where} \,\, z=\mu^{-1}_h(z)\end{array}$$ Use the fact $\widehat{f_N} = N^{-m}  \hat{f}(\xi/N)$ to get 
$$\frac{1}{N^m} D_{\frac{1}{N}}\left (\int_{\R^{m}}\widehat{f}(\frac{\xi}{N}) e^{N(\phi(e^{i\frac{\xi}{N}}|z|^{2})-
\phi(|z|^{2}))}A_{N}(z,ze^{i\frac{\xi}{N}})/B_{h^N}(z, z)d\xi\right), \mbox{where} \,\, z=\mu^{-1}_h(z) $$
Now do the final dilation to get:
\begin{equation}\label{ddd} 
\frac{1}{N^m}\int_{\R^{m}}\widehat{f}(\frac{\xi}{N}) e^{N(\phi(e^{i \frac{\xi}{N} }|\tilde z_N|^{2})-
\phi(|\tilde z_N|^{2}))}A_{N}(\tilde z_N,\tilde z_Ne^{i\frac{\xi}{N}})/B_{h^N}(\tilde z_N,\tilde z_N) d\xi ,\,\, \mbox{where} \,\, \tilde z_N=\mu^{-1}_h(\frac{x}{N})
\end{equation}
To the phase in the integral (\ref{ddd}), recall asymptotic expansion (\ref{dgdffsag}), we have
 \begin{equation}\label{dddhghsd}\begin{array}{lll}&&\phi(e^{iN^{-1}\xi }|\tilde z_N|^{2})-
\phi(|\tilde z_N|^{2})\\ &&\\&=&iN^{-1}\langle \mu_h(\tilde z_N),\xi\rangle-\frac{1}{2}N^{-2}\langle H_{\phi}(\tilde\rho_N)\xi,\xi \rangle +O(N^{-3})\\ &&\\ &=&\langle e^{iN^{-1}\xi } - 1,\frac{ x }{N}\rangle +\frac{1}{2}N^{-2}\langle[M(\frac{x}{N})-H_{\phi}(\tilde \rho_N)]\xi,\xi \rangle+ O(N^{-3})\end{array}\end{equation}
  where in the last step, we replace $iN^{-1}\langle \mu_h(\tilde z_N),\xi\rangle$ by the following Taylor expansion: $$\langle e^{iN^{-1}\xi } - 1,\frac{ x }{N}\rangle= iN^{-1} \langle \frac{x}{N}, \eta\rangle-\frac{1}{2}N^{-2}\langle M(\frac{x}{N})\xi,\xi \rangle+O(N^{-2}) $$ 
where $e^{iN^{-1}\xi } - 1$ denotes the vector $( e^{iN^{-1}\xi_1 } - 1, \ldots, e^{iN^{-1}\xi_m } - 1)$.
 
Thus, we can rewrite (\ref{ddd}) as $$
 \frac{1}{N^m}\int_{\R^{m}}\widehat{f}(\frac{\xi}{N}) e^{\langle e^{i N^{-1}\xi} - 1, x \rangle}e^{\frac{1}{2}N^{-1}\langle[M(\frac{x}{N})-H_{\phi}(\tilde \rho_N)]\xi,\xi \rangle+O(N^{-2})} A_{N}(\tilde z_N,\tilde z_N e^{i \frac{\xi}{N}})/B_{h^N}(\tilde z_N, \tilde z_N) d \xi\label{ssss}$$
Now change variables to $\eta=\xi/N$ to get
\begin{equation}
\int_{\R^{m}}\widehat{f}(\eta) e^{\langle e^{i\eta} - 1, x \rangle}e^{\frac{N}{2}\langle[M(\frac{x}{N})-H_{\phi}(\tilde \rho_N)]\eta,\eta \rangle+O(N^{-2})} A_{N}(\tilde z_N,\tilde z_N e^{i \eta})/B_{h^N}(\tilde z_N, \tilde z_N) d \xi\label{ssss}
 \end{equation}The factor $e^{\frac{N}{2}\langle[M(\frac{x}{N})-H_{\phi}(\tilde \rho_N)]\eta,\eta \rangle+O(N^{-2})}$ admits complete asymptotics. This can be seen from
(\ref{matrixexpansion}) that $M(\frac{x}{N})-H_{\phi}(\tilde \rho_N)$ admits complete asymptotics. And the first two terms are given by:$$e^{\frac{N}{2}\langle[M(\frac{x}{N})-H_{\phi}(\tilde \rho_N)]\eta,\eta \rangle+O(N^{-2})}=1+\frac{1}{2N}\langle M^2(x)\nabla^2 h(0)\eta, \eta \rangle +O(N^{-2})$$
The factor $A_N/B_{h^N}$ admits complete asymptotics,
$$A_N/B_{h^N} =\frac{ N^m(1+a_1(\tilde z_N, \tilde z_Ne^{iN^{-1}\xi})N^{-1}+\cdots )}
{ N^m(1+a_1(\tilde z_N, \tilde z_N)N^{-1}+\cdots )}$$
In the toric case, by the formula in \cite{BBS}, we know that each coefficient $a_j(z, z)$ must be in
the form of $a_j(|z|^2)$; $a_j(z,w)$ is the almost holomorphic extension of $a_j(z, z)$, which implies
that $a_j(z, e^{iN^{-1}\xi}z)$ is a function of $a_j(e^{iN^{-1}\xi}|z|^2)$. Thus we can rewrite,
$$A_N/B_{h^N} =\frac{ N^m(1+a_1(e^{iN^{-1}\xi}|\tilde z_N|^2)N^{-1}+\cdots )}
{ N^m(1+a_1(|\tilde z_N|^2)N^{-1}+\cdots )}$$
We know from (\ref{zexpansion}) that $|\tilde z_N|^2$ admits complete expansion, thus if we apply the Taylor
expansion to each coefficient $a_j$ at $0$, then we get complete asymptotics of $A_N/B_{h^N}$. The
first two terms of complete asymptotics are given by:
$$A_N/B_{h^N}  = 1 + O(N^{-2})$$
Sine both factors $e^{\frac{N}{2}\langle[M(\frac{x}{N})-H_{\phi}(\tilde \rho_N)]\eta,\eta \rangle+O(N^{-2})}$ and $A_N/B_{h^N}$ admit complete asymptotics,
the integral (\ref{ssss}) admits complete asymptotics, i.e., $D_{\frac{1}{N}} S_{h^N}D^{-1}_{\frac{1}{N}}$
admits complete asymptotics.
The leading term of the complete asymptotics is given by:
 $$\int_{\R^{m}}\widehat{f}(\xi) e^{\langle e^{i\eta} - 1,  x\rangle} d\eta= \sum _{|\alpha|=0}^{\infty}e^{-\|x\|}\frac{x^\alpha}{\alpha!}\int_{\R^m} \widehat{f}(\eta)e^{i\langle \alpha, \eta\rangle}d\eta= e^{-\|x\|}\sum _{|\alpha|=0}^{\infty}f(\alpha)\frac{x^\alpha}{\alpha!}=S_{h_{BF}^1}f(x)$$
And the second term which is the coefficient of $N^{-1}$ is given by: $$\begin{array}{lll}b_1(f)(x)&=&\int_{\R^m} \hat f(\eta)e^{\langle e^{i\eta}-1, x\rangle}\left(\langle \frac{1}{2}M^2(x)\nabla^2 h(0)\eta, \eta\rangle \right)d\eta\\ &&\\ &=&\frac{1}{2}\sum_{i,j}a_{ij}(x)\left(e^{-\|x\|}\sum_{\alpha \in \Z^m_+} f_{ij}(\alpha)\frac{x^\alpha}{\alpha!}\right)
\end{array}$$where $(a_{ij}(x))_{i,j}$ denotes the matrix $M^2(x)\nabla ^2h(0)$ and $f_{ij}$ is the partial derivative $\frac{\partial^2 f}{\partial x_i\partial x_j}$.

\bigskip
Now we sketch the proof of the identity (\ref{wall}) where we only rescale variables $ x'$ and fix the rest $x''$. In this case, we use the following Slice-orbit coordinates \cite{SoZ}: Assume $\{F_j\}_{j=1}^{m}$ are all facets which have a common vertex at $0$. Let $F=\cap_{j=1}^k F_j$ be a $(m-k)$-dimensional (open) face. Then by angle-action coordinates, we can split the torus action into two parts $\mathbb{T}^{m}=\mathbb{T}^{k}\times \mathbb{T}^{m-k}$, where $\mathbb{T}^{k}$ is the stabilizer at $\mu_h^{-1}(F)$, and we can split all coordinates  into two parts corresponding to $\mathbb{T}^{k}$ and $\mathbb{T}^{m-k}$ action: We can split $z=(z^{'},z^{''})\in \mathbb{C}^{k}\times \mathbb{C}^{m-k}$ where $(0,z^{''})$ is a local coordinate of the submanifold $\mu_h^{-1}(\bar F)$, where $\bar F$ is the closure of $F$. In addition, we have the decomposition of the moment map $\mu_h(z)=(\mu_h^{'}(z),\mu_h^{''}(z))$, where $\mu_h^{''}(z)$ (resp. $\mu_h^{'}(z)$) is the moment map for the Hamiltonian $\mathbb{T}^{m-k}$ (resp. $\mathbb{T}^{k}$) action on the \kahler toric submanifold with $z^{'}=0$ (resp. $z^{''}=0$).   

 Denote $\varsigma=(\eta, \xi)$ and 
 $\mu_h(z)=(\frac{x'}{N},x'')$, follow the steps in (\ref{ddd}),
$$\begin{array}{lll}
&&(\hat D_{\frac{1}{N}}S_{h^N}\hat D^{-1}_{\frac{1}{N}})f(x',x'')\\&&\\&=&N^{-k}\int_{\R^{m}}\widehat{f}(\frac{\eta}{N},\xi) e^{N(\phi(e^{iN^{-1}\varsigma }| z|^{2})-
\phi(|z|^{2}))}A_{N}( z,  ze^{i N^{-1} \varsigma})/B_{h^N}( z,  z) d\xi d\eta \end{array}$$
To the phase of this integral, recall the expansion (\ref{dgdffsag}), we have
$$\begin{array}{lll}\phi(e^{iN^{-1} \varsigma}|\mu_h ^{-1}(\frac{x}{N})|^{2})-
\phi(|\mu_h ^{-1}(\frac{x}{N})|^{2})&=&iN^{-1}\langle \mu_h( z), \varsigma\rangle-\frac{1}{2}N^{-2}\langle H_{\phi}(\rho) \varsigma, \varsigma \rangle +O(N^{-3})\\ &&\\ &=&iN^{-1}\langle\mu''( z),\xi\rangle+\langle\mu'( z),e^{iN^{-1}\eta}-1\rangle+O(N^{-2})\end{array}$$
 where in the last step, we follow the steps in (\ref{dddhghsd}).
 
 Change variable $\eta/N$ to $\zeta$ and use the asymptotics of $A_N/B_{h^N}$ to get
 $$\begin{array}{lll}(\hat D_{\frac{1}{N}}S_{h^N}\hat D^{-1}_{\frac{1}{N}})f(x',x'')&=& N^{-k}\int_{\R^{k}}\hat{f}(\frac{\eta}{N},\xi) e^{i\langle x'',\xi\rangle}e^{\langle x^{'}, e^{iN^{-1}\eta}-1\rangle}d\xi d\eta+O(N^{-1})\\&&\\&=& \int_{\R^{k}}\hat{f}(\zeta,\xi) e^{i\langle x'',\xi\rangle}e^{\langle x^{'}, e^{i\zeta}-1\rangle}d\xi d\zeta+O(N^{-1})\\&&\\&=&e^{-\|x^{'}\|}\sum _{\alpha \in \Z^k_+}\frac{x^{'\alpha}}{\alpha!}f_{x^{''}}(\alpha)+O(N^{-1})\\&&\\&=&S_{h_{BF}^1}(f_{x''})(x')+O(N^{-1})\end{array}$$
where $f_{x''}(x')$ means we fix the variable $x''$ and consider $f(x',x'')$ as a function of $x'$ and $S_{h_{BF}^1}(f_{x''})(x')$ is the Szasz analytic functions with respect to variables $x'$. 
\end{proof}
\section{Examples of generalized Szasz analytic functions}\label{example}
\subsection{\label{E}$B^m $ with the Bergman metric}
In this subsection, we further compute the generalized Szasz analytic functions for the canonical line bundle over $B^m$ with the Bergman metric. 
We denote the unit ball as:
$$B^m=\left\{(z_1,\ldots, z_m)\in \C^m|\|z\|^{2}=|z_1|^{2}+\cdots +|z_m|^{2}< 1 \right\}$$
There is a classical \kahler metric on $B^m$, under the logarithmic coordinates, this \kahler potential is   \begin{equation}\varphi(z) =-\log (1-\|z\|^{2})\end{equation} Thus we denote $$\phi(\rho)=-\log({1-e^{\rho_1}-\cdots-e^{\rho_m}})$$
So the moment map is: \begin{equation}\mu_{h_B}(z_{1}, . . . , z_{m}) = \nabla _{\rho}\varphi(\rho) =(\frac{|z_{1}|^{2}}{1-\|z\|^{2}},\cdots , \frac{|z_{m}|^{2}}{1-\|z\|^{2}}) \label{eieie}
\end{equation}
Thus the moment polyhedral set is $ \R_+^m$.
Denote $\|x\|=\sum_{j=1}^{m}x_{j}$ where each $x_j>0$, then the sympletic potential is
 \begin{equation}u_B(x)=\sum_{j=1}^{m}x_{j}\log x_{j}-(1+\|x\|)\log(1+\|x\|) \label{efffieie}
\end{equation}
Now let's compute the Bergman kernel for the line bundle $L^N \to B^m$ with the Bergman metric where $L$ is the canonical line bundle . First, we have the Bergman kernel for the line bundle $\mathcal{O}(N) \to \mathbb{CP}^m$ with the Fubini-Study metric \cite{Z1}:
$$B_{h^{N}_{FS}}(z,w)=\frac{(N+m)!}{N!}(\frac{(1+\langle z,w\rangle)}{\sqrt{1+\|z\|^{2}}\sqrt{1+\|w\|^{2}}})^{N}$$
where $\langle z,w\rangle =z_1\bar w_1+\cdots+z_m\bar w_m$.

 Since $B^m$ is dual to $\mathbb{CP}^{m}$ as a symmetric space in the sense of S. Helgason \cite{H}, its Bergman kernel
 should replace the term $1+\|z\|^{2}$ in $B_{h^{N}_{FS}}$  by $1-\|z\|^{2}$. The hyperplane line bundle $\mathcal{O}(1) \to \mathbb{CP}^{m}$ carries a natural positive Fubini-Study metric, then the canonical bundle $\mathcal{O}(-m-1)\rightarrow \mathbb{CP}^{m}$ will be negative curved. By the duality again, the reverse is true on $B^m$. Hence, we also need to change the sign of the exponent to get \cite{K} $$B_{h_B^{N}}(z,w)=\frac{(N+m)!}{N!}(\frac{(1-\langle z,w\rangle)}{\sqrt{1-\|z\|^{2}}\sqrt{1-\|w\|^{2}}})^{-N}$$
On the diagonal we have \begin{equation}B_{h_B^{N}}(z,z)=\frac{(N+m)!}{N!}\label{kkjhgfjl}\end{equation}
 For simplicity, let's compute the generalized Szasz analytic functions for the unit disc when $m=1$, the same results are valid for higher dimension. In this case, by simple computation \begin{equation}\|z^{j}\|^{2}_{h_B^{N}}=\frac{j!\Gamma(N-1)}{\Gamma( N+j)}\end{equation} where $\Gamma(n)=(n-1)!$. Thus the orthonormal basis is \begin{equation}\left\{\sqrt{\frac{\Gamma( N+j)}{j!\Gamma(N-1)}}z^{j}, \,\,\ j \geq 0, \,\,\,N\geq 2\right\}\label{kjkjhl}\end{equation}

Hence, \begin{equation}\label{dfggdfddf}\begin{array}{lll}S_{h_B^{N}}(f)(x) &=& \frac{1}{B_{h_B^{N}}(z,z)}\sum _{j = 0}^{\infty}f(\frac{j}{N})\frac{e^{N(u_B(x)+\langle\frac{j}{N}-x,\nabla u_B(x)\rangle)}}{\|z^{j}\|^{2}_{h_B^{N}}} \\ &&\\ &=& \frac{(1+x)^{-N}}{B_{h_B^{N}}(z,z)}\sum _{j=0}^{\infty}f(\frac{j}{N})\frac{(\frac{x}{1+x})^{j}}{\|z^{j}\|^{2}_{h_B^{N}}}\\&&\\ &=&\frac{N-1}{N+1}\cdot (1+x)^{-N}\sum _{j=0}^{\infty}(N)_{j}f(\frac{j}{N})\frac{(\frac{x}{1+x})^{j}}{j!}\\&&\\ &=& (1+x)^{-N}\sum _{j=0}^{\infty}(N)_{j}f(\frac{j}{N})\frac{(\frac{x}{1+x})^{j}}{j!}+O(\frac{1}{N})\end{array}\end{equation}
where $(N)_{j}=N(N+1)\cdots (N+j-1)$.

Hence, in view of Theorem \ref{BBa}, we have the following estimate :
$$(1+x)^{-N}\sum _{j=0}^{\infty}(N)_{j}f(\frac{j}{N})\frac{(\frac{x}{1+x})^{j}}{j!} \to f(x) ,\,\,\,as \,\,\, N \to \infty$$

\subsection{ Probabilistic interpretation} In this subsection, we give some probabilistic interpretation of the generalized Szasz analytic functions. It's suffice to consider the $1$-dimensional case.

\subsubsection{$\mathbb{CP}^1$ and binomial distribution}
Recall in \cite{Z1} that if we take the \kahler toric manifold as $\mathbb{CP}^1$ with Fubini-Study metric, then the generalized Szasz analytic function $S_{h_{FS}^{N}}(f)$ becomes the classical Bernstein polynomial , where $f\in C^{\infty}([0,1])$.

Assume $X$ is the independent and identically distributed Bernoulli trial which yields success with probability $x \in (0,1)$. $X$ has the $N+1$ values $j=0,1, \cdots, N$ with the probability $$P_{\mathbb{CP}^1, N}(X=j)= {N \choose j}x^j(1-x)^{N-j}. \,\,\,$$Thus we have
$$\begin{array} {lll} E(f(\frac{X}{N}))=\sum_{j=0}^N f(\frac{j}{N}){N \choose j} x^j(1-x)^{N-j}=S_{h_{FS}^{N}}(f)\end{array} $$
\subsubsection{$\C$ and Poisson distribution}
Recall that if we take our \kahler toric manifold as $\C$ with the Bargmann-Fock metric, then the generalized Szasz analytic function  becomes the classical Szasz analytic function $S_{h_{BF}^{N}}(f)$, where $f\in C_0^{\infty}([0,\infty))$.

The Poisson distribution $X$ describes the number of independent events occurring in a fixed period of time. $X$ has the infinite number values $j=0,1,2 \cdots$ with the probability $$P_{\C, N}(X=j)=e^{-Nx}\frac{(Nx)^j}{j!}\,\,\, $$where $Nx>0$ representing the average arrival rate during the period. Thus we have $$\begin{array} {lll} E(f(\frac{X}{N}))=\sum_{j=0}^ {\infty} f(\frac{j}{N})e^{-Nx}\frac{(Nx)^j}{j!}=S_{h_{BF}^{N}}(f)\end{array} $$
\subsubsection{Disk and Pascal distribution}
Recall that the Pascal distribution is the probability distribution of failures before the first success in a series of independent and identically distributed Bernoulli trials.

Assume the independent and identically distributed Bernoulli trial yields success with probability $p \in (0,1)$. And denote $q=1-p$, then for any integer $j\geq N$,
denote $P_{B,N}(X=j)$ as the probability of the event that there are exact $N-1$ successful trials among $X_{1},X_{2},\ldots ,X_{j-1}$ and $X_{j}$ is the successful trial. Hence we have $$P_{B,N}(X=j)={j-1 \choose N-1}p^{N}q^{j-N},\,\,\, j=N,N+1,N+2 \cdots$$
\begin{theo}If $f(x)$ is smooth with compact support on $[0,\infty)$, then we have:
\begin{equation} S_{h_B^{N}}(f)(x)= E(f(\frac{X}{N}))+ O(\frac{1}{N})
\end{equation}
where $X$ is Pascal distribution and the probability of the successful trial is $p=\frac{1}{1+x}$.
\end{theo}
\begin{proof} By the fact $p=\frac{1}{1+x}$ and $q=\frac{x}{1+x}$, we have $$\begin{array}{lll}E(f(\frac{X}{N})) &=& \sum _{j\geq N}f(\frac{j}{N}){j-1 \choose N-1}p^{N}q^{j-N}\\&&\\&=&p^{N}\sum_{j=0}^{\infty}f(\frac{j}{N}){N+j-1 \choose N-1}q^{j} \\&&\\&=& (1+x)^{-N}\sum_{j}(N)_{j}f(\frac{j}{N})\frac{(\frac{x}{1+x})^{j}}{j!}\\&&\\&=&  S_{h_B^{N}}(f)(x)- O(\frac{1}{N})\end{array}$$
in the last step we use the estimate (\ref{dfggdfddf}).
\end{proof}

\end{document}